\documentclass[11pt,onecolumn]{article}

\usepackage{geometry}
\usepackage{mathtools}

\usepackage{amsmath,mathrsfs, textcomp, eufrak, mathtools}
\usepackage{amssymb,enumerate,amsthm}
\usepackage{pstricks,pst-plot,psfrag}
\usepackage[all]{xy}
\usepackage{graphicx,subfigure,xspace,bm}
\usepackage{etex}
\usepackage{todonotes}
\usepackage[normalem]{ulem}

\usepackage[lined,linesnumbered,algoruled,noend]{algorithm2e}

\newtheorem{theorem}{Theorem}
\newtheorem{definition}{Definition}
\newtheorem{assumption}{Assumption}
\newtheorem{lemma}{Lemma}

\newtheorem{remark}{Remark}

\newtheorem{corollary}{Corollary}
\newtheorem{proposition}{Proposition}  
\newcommand{\real}{{\mathbb{R}}}
\newcommand{\integers}{\mathbb{Z}}

\newcommand\redsout{\bgroup\markoverwith{\textcolor{red}{\rule[0ex]{2pt}{5pt}}}\ULon}

\newcommand{\ones}{\mathbf{1}}
\newcommand{\zeros}{\mathbf{0}}

\definecolor{darkgreen}{rgb}{0,0.5,0}

\newcommand{\oprocendsymbol}{\hbox{$\bullet$}}
\newcommand{\oprocend}{\relax\ifmmode\else\unskip\hfill\fi\oprocendsymbol}

\newcommand{\eps}{\epsilon}

\newcommand{\map}[3]{#1:#2 \rightarrow #3}

\newcommand\upscr[2]{#1^{\textup{#2}}}

\newcommand{\R}{\real}

\newcommand{\dout}{\upscr{d}{out}}

\newcommand{\Nin}{\upscr{N}{in}}
\newcommand{\Nout}{\upscr{N}{out}}
\renewcommand{\S}{\mathcal{S}}

\newcommand{\G}{\mathcal{G}}

\newcommand{\B}{\mathcal{B}}
\newcommand{\A}{\mathbb{A}}

\DeclareMathOperator*{\argmin}{arg\,min}

\renewcommand{\Pr}{\mathbb{P}}

\newcommand{\Ev}{\mathcal{A}}
\newcommand{\Ec}{\mathcal{C}}
\newcommand{\Es}{\mathcal{E}}
\newcommand{\Ex}{\mathbb{E}}
\newcommand{\V}{\mathcal{V}}
\newcommand{\Q}{\mathbb{Q}}

\usepackage{cancel}

\newcommand{\myclearpage}{\clearpage}
\renewcommand{\myclearpage}{}

\definecolor{BBlue}{cmyk}{.98,0.10,0,.25}

\begin{document}
\title{Push-sum on random graphs}
\author{Pouya Rezaienia, Bahman Gharesifard,\\ Tam{\'a}s Linder, and Behrouz Touri\thanks{The first three authors are with the Department of Mathematics and Statistics at Queen's University, Kingston, ON, Canada. The last author is with the Department of Electrical and Computer Engineering at the University of California, San Diego.}}
\maketitle
\begin{abstract}
In this paper, we study the problem of achieving average consensus over a random time-varying sequence of directed graphs by extending the class of so-called push-sum algorithms to such random scenarios. Provided that an ergodicity notion, which we term the directed infinite flow property, holds and the auxiliary states of agents are uniformly bounded away from zero infinitely often, we prove the almost sure convergence of the evolutions of this class of algorithms to the average of initial states. Moreover, for a random sequence of graphs generated using a time-varying $B$-irreducible probability matrix, we establish convergence rates for the proposed push-sum algorithm.
\end{abstract}

\section{Introduction}\label{sec:intro}
Many distributed algorithms, executed with limited information over a network of agents, rely on estimating the average value of the initial state of the individual agents. These include the distributed optimization protocols~\cite{JNT-DPB-MA:86, MR-RN:04,LX-SB:06, AN-AO:09,PW-MDL:09,AN-AO-PAP:10,BJ-MR-MJ:09,JW-NE:10,JW-NE:11, BG-JC:14-tac,AN-AO:15-tac}, distributed regret minimization algorithms in machine learning~\cite{MA-BG-TL:15-tcsn}, and dynamics for fusion of information in sensor networks~\cite{KIT-SL-MGR:12}. There is a large body of work devoted to the average consensus problem, starting with the pioneering work~\cite{DK-AD-JG:03}, where the so-called \emph{push-sum algorithm} is first introduced. The key differentiating factor of the push-sum algorithm from consensus dynamics is that it takes advantage of a paralleled scalar-valued agreement dynamics, initiated uniformly across the agents, that tracks the imbalances of the network and adjusts for them when estimating the consensus value.

In addition to the earlier work~\cite{DK-AD-JG:03}, several recent papers have studied the problem of average consensus, see for example~\cite{ADDG-CNH:13}, where other classes of algorithms based on weight adaptation are considered, ensuring convergence to the average on fixed directed graphs. The study of convergence properties of push-sum algorithms on  time-varying deterministic sequences of directed graphs, to best of our knowledge, was initiated in~\cite{FB-VB-PT-JT-MV:10} and extended in~\cite{AN-AO:15-tac}, where push-sum protocols are intricately utilized to prove the convergence of a class of distributed optimization protocols on a sequence of time-varying directed graphs. The key assumption in~\cite{AN-AO:15-tac} is the $ B $-connectedness of the sequence, which means that in any window of size $ B $ the union of the underlying directed graphs over time is strongly connected. As we demonstrate, a by product of our work in deterministic settings is the generalization of the sequences on which the convergence of the push-sum algorithms is valid to the ones which satisfy the infinite flow property; in this sense, this extension mimics the properties required for the convergence of consensus dynamics, along the lines of~\cite{BT:12-book}.

This paper is concerned with the problem of average consensus for scenarios where communication between nodes is time-varying and possibly random. The convergence properties of consensus dynamics on random sequences of directed graphs are by this time well-established, see for example~\cite{BT:12-book,touri2014endogenous,touri2014product}.  Average consensus on random graphs has also been studied in~\cite{FB-VB-PT-JT-MV:10}, under the assumption that the corresponding random sequence of stochastic  matrices is \emph{stationary} and ergodic with positive  diagonals and irreducible expectation. One of our main objectives in this work is to extend these result to more general sequences of random stochastic matrices, in particular, beyond stationary.  More importantly, to best of our knowledge, we establish for the first time convergence rates for the push-sum algorithms on random sequences of directed graphs.

The remainder of this paper is organized as follows. Section~\ref{sec:prelim} contains mathematical preliminaries. In Section~\ref{sec:statement}, we give a formal description of our consensus problem. In Section~\ref{sec:randompushsum}, we describe the push-sum algorithm. Section~\ref{sec:ergodicity} studies the ergodicity of row-stochastic matrices, and Section~\ref{sec:convergenceofpushsum} contains our main convergence results. In Section~\ref{sec:convergencerate}, we derive convergence rates for the push-sum algorithm for a class of random column-stochastic matrices. Finally, we gather our conclusions and ideas for future directions in Section~\ref{sec:conclusion}.

\section{Mathematical Preliminaries}\label{sec:prelim}
We start with introducing some notational conventions.  Let $ \real $ and $ \integers $ denote the set of real and integer numbers, respectively, and let $ \real_{\geq 0} $ and $ \integers_{\geq 0} $ denote the set of non-negative real numbers and integers, respectively. For a set $\A$, we write $S\subset \A$ if $S$ is a proper subset of $\A$, and we call the empty set and $\A$ trivial subsets of $\A$. The complement of $S$ is denoted by $\bar{S}$. Let $ |S| $ denote the cardinality of a finite set $ S $. We view all vectors in $\real^n$ as column vectors, where $n\in \integers_{\geq 0}$. We denote by $ \|\cdot\| $, $ \|\cdot\|_1 $ and $ \|\cdot\|_{\infty} $, the standard Euclidean norm, the $1$-norm, and the infinity norm on $ \real^n $, respectively. The $i$th unit vector in $ \real^n $, whose $ i $th component is $ 1 $ and all other components are $0$, is denoted by $ e_i $. We will also use the short-hand notation $ \ones_n=(1,\ldots,1)^T $ and $ \zeros_n=(0,\ldots,0)^T \in \real^n $. A vector $v$ is stochastic if its elements are nonnegative real numbers that sum to $1$. We use $ \real_{\geq 0}^{n\times n}$ to denote the set of $n\times n$ non-negative real-valued matrices. A matrix $A\in \real_{\geq 0}^{n\times n}$ is row-stochastic (column-stochastic) if each of its rows (columns) sums to 1.  For a given $A\in \real_{\geq 0}^{n\times n}$ and any nontrivial $S\subset [n]$, we let ${A_{S\bar{S}} = \sum_{i\in S, j\in \bar{S}}A_{ij}}$. The notation $A'$ and $v'$ will refer to the transpose of the matrix $A$ and the vector $v$, respectively. A positive matrix is a real matrix all of whose elements are positive. Finally, $A_i$ denotes the $i$th row of matrix $A$ and $A^j$ denotes the $j$th column of $A$.
\subsection{Graph theory}
A (weighted) \emph{directed graph} $\G=(\V,\Es,A) $ consists of a node set \sloppy ${\V = \{v_1,v_2,\ldots,v_n\}}$, an edge set $ \Es \subseteq \V\times \V$, and a weighted \emph{adjacency matrix} $ {A \in \real^{n\times n}_{\geq0}} $, with $ a_{ji}>0 $ if and only if $ (v_i,v_j)\in \Es $, in which case we say that $v_i$ is connected to $v_j$. Similarly, given a matrix $ A \in \real^{n\times n}_{\geq0} $, one can associate to $ A $ a directed graph $ \G=(\V,\Es) $, where $ (v_i,v_j)\in \Es $ if and only if $ a_{ji}>0 $, and hence $ A $  is the corresponding adjacency matrix for $ \G $. The in-neighbors and the out-neighbors of $v_i$ are the set of nodes ${\Nin_i= \{j\in [n]: a_{ij}>0\}}$ and ${\Nout_i= \{j\in [n]: a_{ji}>0\}}$, respectively. The out-degree of $v_i$ is ${\dout_i=|\Nout_i|}$. A path is a sequence of nodes connected by edges.  A directed graph is \emph{strongly connected} if there is a path between any pair of nodes. A directed graph is \emph{complete} if every pair of distinct vertices is connected by an edge. If the directed graph $\G=(\V,\Es,A) $ is strongly connected, we say that $A$ is irreducible.
\subsection{Sequences of random stochastic matrices}
Let $\S^{+}_{n}$ be the set of $n\times n$ column-stochastic matrices that have positive diagonal entries, and let $\mathcal{F}_{\S^{+}_{n}}$ denote the Borel $\sigma$-algebra on $\S^{+}_{n}$. Given a probability space $(\Omega, \B, \mu)$, a measurable function \sloppy $ {\map{W}{(\Omega, \B, \mu)}{(\S^{+}_n, \mathcal{F}_{\S^{+}_{n}})}} $ is called a random column-stochastic matrix, and a sequence $ \{W(t)\} $ of such measurable functions on $(\Omega, \B, \mu)$ is called a random column-stochastic matrix sequence; throughout, we assume that $t\in \integers_{\geq 0}$. Note that for any $ \omega \in \Omega $, one can associate a sequence of directed graphs $\{\G(t)(\omega)\}$ to $\{W(t)(\omega)\} $, where $(v_i,v_j)\in \Es(t)(\omega)$ if and only if $W_{ji}(t)(\omega)>0$. This in turn defines a sequence of random directed graphs on $ \V=\{v_1,\ldots, v_n\} $, which we denote by $\{\G(t)\}$. 
\myclearpage
\section{Problem Statement}\label{sec:statement}
Consider a network of nodes $ \V=\{v_1,v_2,\ldots, v_n\} $, where node $ v_i \in \V$ has an initial state (or opinion) $x_i(0)\in \real$; the assumption that this initial state is a scalar is without loss of generality, and our treatment can easily be extended to the vector case. The objective of each node is to achieve \emph{average consensus}; that is to compute the average $ {\bar{x} =\frac{1}{n}\sum_{i=1}^{n}x_i(0)} $ with the constraint that only limited exchange of information between nodes is permitted. The communication layer between nodes at each time $t\geq 0$ is specified by a sequence of random directed graphs $ \{\G(t)\} $, where $ \G(t) = (\V, \Es(t), W(t)) $. Specifically, at each time $ t $, node $ v_i $ updates its value based on the values of its in-neighbors $ v_j\in \Nin_i (t)$, where ${\Nin_i(t)=\{v_j\in \V:W_{ij}(t)>0\}}$.
One standing assumption throughout this paper is that each node knows its out-degree at every time $ t $; this assumption is indeed necessary, as shown in~\cite{JH-JT:15}.
Our main objective is to show that the class of so-called push-sum algorithms can be used to achieve average consensus at every node, under the assumption that the communication network is random. This key point distinguishes our work from the existing results in the literature~\cite{DK-AD-JG:03},~\cite{AN-AO:15-tac},~\cite{ADDG-CNH:13}. Another key objective that we pursue in this paper is to obtain rates of convergence for such algorithms. We start our treatment with reviewing the push-sum algorithm. 

\myclearpage

\section{Random Push-Sum}\label{sec:randompushsum}

Consider a network of nodes $ \V=\{v_1,v_2,\ldots, v_n\} $, where node $ v_i \in \V$ has an initial state (or opinion) $x_i(0)\in \real$. The push-sum algorithm, proposed originally in~\cite{DK-AD-JG:03}, is defined as follows. Each node $ v_i $ maintains and updates, at each time ${t\geq 0}$, two state variables $ x_i(t) $ and $ y_i(t) $. The first state variable is initialized to $ x_i(0) $ and the second one is initialized to $ y_i(0)=1 $, for all $i\in [n]$. At time $t\geq0$, node $v_i$ sends $\frac{x_i(t)}{\dout_i(t)}$ and $\frac{y_i(t)}{\dout_i(t)}$ to its out-neighbors in the random directed graph $\G(t) = (\V, \Es(t), W(t))$, which we assume to contain self-loops at each node for all $t\geq 0$. At time $ (t+1) $, node $ v_i $ updates its state variables according to 
\begin{align}\label{eqn:main-algo}
x_i(t+1) &=\sum_{j\in N_i^{in}(t)}\frac{x_j(t)}{\dout_j(t)},\cr
y_i(t+1) &= \sum_{j\in N_i^{in}(t)}\frac{y_j(t)}{\dout_j(t)}.
\end{align}
It is useful to define another auxiliary variable ${z_i(t+1) = \frac{x_i(t+1)}{y_i(t+1)}}$; as we will show later,  $z_i(t+1) $ is the estimate by node $ v_i $ of the average $ \bar{x} $. One can rewrite this algorithm in a vector form; let the column-stochastic matrix $W(t)$ to be a function of $\Es(t)$ with entries
\begin{equation}\label{eqn:W}
W_{ij}(t) = 
\begin{cases} \frac{1}{\dout_j(t)} &\text{if } j\in \Nin_i(t),\\
0&\text{otherwise.} 
\end{cases}
\end{equation}
Using these weighted adjacency matrices, for every $t\geq 0$, we can rewrite the dynamics~\eqref{eqn:main-algo} as
\begin{align}\label{eqn:mtxalgorithm}
x(t+1) &= W(t)x(t),\cr
y(t+1) &= W(t)y(t),
\end{align}
where 
\begin{align*}
x(t)&= (x_1(t), \ldots, x_n(t))', \cr
y(t)&= (y_1(t), \ldots, y_n(t))'.
\end{align*}

\myclearpage
\section{Ergodicity}\label{sec:ergodicity}
In this section, we establish some important auxiliary results regarding the convergence of products of matrices which satisfy the so-called directed infinite flow property (c.f. Definition~\ref{def:inf-flow}). We study the products of a class of matrices in a deterministic setting, which we then use to study the push-sum algorithm in the next section. We start by some definitions.
\begin{definition}[Ergodicity~\cite{SC-ES:77},~\cite{BT:12-book}]
Let $ \{A(t)\} $ be a sequence of row-stochastic matrices, and for $ t \geq s \geq 0$, let $ A(t:s)$ denote the product
\begin{align}\label{eqn:matrixprod}
A(t:s) = A(t)A(t-1)\cdots A(s),
\end{align}
where $ A(s:s) =A(s) $. The sequence $ \{A(t)\} $ is said to be \textit{weakly ergodic}, if for all $i,j,l \in [n]$ and any $s \geq 0$, $\lim_{t\rightarrow \infty} \left(A_{il}(t:s) - A_{jl}(t:s)\right) =0$. The sequence is said to be \textit{strongly ergodic} if $\lim_{t\rightarrow \infty}  A(t:s) = \ones_n v'(s)$ for any $ s\geq 0 $, where $ v(s) \in \R^n $ is a stochastic vector.  
\end{definition}
It can be shown that weak ergodicity and strong ergodicity are equivalent~\cite[Theorem~1]{SC-ES:77}. We will simply call such a sequence of row-stochastic matrices ergodic.

We first establish a sufficient condition for ergodicity of a sequence of row-stochastic matrices, Proposition~\ref{prop:shrink}, which we subsequently use in our convergence result for the push-sum algorithm. For this reason, we consider the following dynamical system:
\begin{align}\label{eqn:dynamicalsystem}
x(t+1) = A(t)x(t), \qquad \text{for all } t\geq 0.
\end{align}
Let us start by two key definitions. 
\begin{definition}[Strong Aperiodicity~\cite{BT:12-book}]\label{def:fdbk}
We say that a sequence of matrices $\{A(t)\}$ is \textit{strongly aperiodic} if there exists $\gamma>0$ such that ${A_{ii}(t)\geq \gamma}$, for all $t\geq0$ and $i\in [n]$.
\end{definition}

Motivated by the \textit{infinite flow property}~\cite[Definition 3.2.]{BT:12-book}, we provide the following definition.
\begin{definition}[Directed Infinite Flow Property]\label{def:inf-flow}
We say that a sequence of matrices $ \{A(t)\} $ has the \textit{directed infinite flow property} if for any non-trivial $ S\subset[n] $, $ {\sum_{t=0}^{\infty}A_{S\bar{S}}(t) = \infty} $.
\end{definition}
Consider now a sequence of matrices $\{A(t)\}$ that is strongly aperiodic and has the directed infinite flow property. Let ${k_0=0}$, and for any $q\geq1$, define
\begin{align}\label{def:kq}
k_q = \argmin_{t'> k_{q-1}} \left(\min_{S\subset [n]}\sum_{t=k_{q-1}}^{t'-1}A_{S\bar{S}}(t)> 0\right).
\end{align}
Note that $k_q$ is the minimal time instance after $k_{q-1}$, such that there is nonzero information flow between any non-trivial subset of $\V$ and its complement; consequently, the directed graph associated with the product ${A(k_{q}-1)A(k_{q}-2)\cdots A(k_{q-1})}$ is strongly connected.
\begin{proposition}
If a sequence of matrices $\{A(t)\}$ has the directed infinite flow property, $k_q$ is finite for all $ q \geq 0 $.
\end{proposition}
\begin{proof}
Suppose that $k_q$ is not finite for some ${q\geq 0}$. Then, using~\eqref{def:kq}, there exists a non-trivial subset $S\subset [n]$ such that ${\sum_{t=k_{q-1}}^{\infty}A_{S\bar{S}}(t)= 0}$. This implies that $\sum_{t=0}^{\infty}A_{S\bar{S}}(t)< \infty$, which contradicts the assumption that $\{A(t)\}$ has the directed infinite flow property.
\end{proof}
To establish convergence results for the products of row-stochastic matrices satisfying Proposition~\ref{def:inf-flow}, we argue that in each time window where the underlying directed graph becomes strongly connected for $n$ times, i.e., after $k_{qn} - k_{(q-1)n}$ time steps for some $q$, \emph{significant mixing} will occur. To formalize this statement, let $\ell_0 =0$ and
\begin{align}\label{def:lq}
\ell_q = k_{qn}-k_{(q-1)n},
\end{align}
for $q\geq 1$. For $t>s\geq 0$, we also define
\[\Q_{t,s}= \{q: s\leq k_{(q-1)n},k_{qn}\leq t\}.\]
We are now ready to state our first result.
\begin{proposition}\label{prop:shrink}
Consider the dynamics~\eqref{eqn:dynamicalsystem}, where the sequence of row-stochastic matrices $\{A(t)\}$ is such that $A'(t)$ satisfies~\eqref{eqn:W}. Suppose, additionally, that  $\{A(t)\}$
is strongly aperiodic and has the directed infinite flow property.
Then,
\begin{enumerate}[(i)]
\item
there is a vector $\phi(s) \in \real^{n}$ such that, for all $i,j \in [n]$ and $t\geq s$,
\begin{align*}
\biggl|[A(t:s)]_{ij} - \phi_j(s)\biggl| \leq \Lambda_{t,s},
\end{align*}
where $\Lambda_{t,s}=\prod_{q\in \Q_{t,s}} \lambda_q$ and $\lambda_q = \left(1-\frac{1}{n^{\ell_q}}\right) \in (0,1)$;
\item
if, for the sequence $\{\ell_q\}$ associated with $\{A(t)\}$, we have
\begin{align}\label{eqn:suminfty}
\sum_{q=1}^{\infty}\frac{1}{n^{\ell_q}}= \infty,
\end{align}
then the sequence $\{A(t)\}$ is ergodic.
\end{enumerate}
\end{proposition}
\begin{proof}
We start by proving the first statement.
By definition of $k_q$, we know that for all $q\geq 0$, ${A(k_{q+1}-1:k_q)}$ is irreducible. Since each $A(t)$ is strongly aperiodic, by Lemma~\ref{lemma:compgraph}, the matrix
\begin{multline*}
A(k_{n(q+1)} -1:k_{nq})\\
= A(k_{n(q+1)}-1:k_{n(q+1)-1})\times\cdots\times A(k_{nq+2}-1:k_{nq+1})\\
\times A(k_{nq+1}-1:k_{nq}),
\end{multline*}
which is the product of $n$ irreducible matrices, is positive for all $q\geq 0$.
Hence, by Lemma~\ref{lemma:Asu}~(ii), for all $i,j\in [n]$, we have
\begin{align*}
[A(k_{n(q+1)} -1:k_{nq})]_{ij}\geq \frac{1}{n^{k_{n(q+1)} -k_{nq}} }= \frac{1}{n^{l_{q+1}}}.
\end{align*}
Now, since $A(t:s) = A(t:s) I_n$ and for all $j\in [n]$, \sloppy ${\max_{i\in[n]}[I_n]_{ij}-\min_{i\in[n]}[I_n]_{ij}= 1 }$, using~\cite[Lemma~3]{JH-MSB:58}, we obtain
\begin{align}\label{eqn:m-m}
\max_{i\in [n]} [A(t:s)]_{ij} - \min_{i\in [n]} [A(t:s)]_{ij} \leq \Lambda_{t,s}.
\end{align}
Note that if we let $ \phi_j(s) = \min_{i\in [n]}A_{ij}(t:s) $ for all $ j \in [n] $, we have 
\begin{align}\label{eqn:A-phi-ineq}
\biggl| [A(t:s)]_{ij} - \phi_j(s) \biggl| \leq \max_{i\in [n]} [A(t:s)]_{ij} - \min_{i\in [n]} [A(t:s)]_{ij}.
\end{align}
Using~\eqref{eqn:m-m} and \eqref{eqn:A-phi-ineq}, we conclude that
\begin{align*}
\biggl|[A(t:s)]_{ij} - \phi_j(s)\biggl| \leq \Lambda_{t,s},
\end{align*}
for all $ i,j\in [n] $.

We next prove part (ii); since $\lambda_q  \in \left(0,1\right)$ for all $q\geq 1$, we have that $\ln\left(\lambda_q\right) \leq\frac{-1}{n^{\ell_q}}$, where we have used the fact that $\ln(\zeta)\leq\zeta-1$ for all $\zeta>0$. This implies
\begin{align}
\sum_{q=1}^{\infty}\ln\left(\lambda_q\right)&\leq-\sum_{q=1}^{\infty}\frac{1}{n^{\ell_q}}.\label{eqn:aux-a}
\end{align}
On the other hand, we have 
\begin{align*}
\lim_{t\rightarrow\infty}\Lambda_{t,0}=\lim_{t\rightarrow\infty}\prod_{q\in \Q_{t,0}}\lambda_q= \lim_{t\rightarrow\infty}\exp\left(\sum_{q\in \Q_{t,0}}\ln\left(\lambda_q\right) \right).
\end{align*}
The definition of the sets $\Q_{t,s}$ implies that we can write the right hand side as $\exp\left({\sum_{q=1}^{\infty}\ln\left(\lambda_q\right)}\right)$, which gives
\begin{align*}
\lim_{t\rightarrow\infty}\Lambda_{t,0} = \exp\left(\sum_{q=1}^{\infty}\ln\left(\lambda_q\right)\right) = 0,
\end{align*}
where the last equality follows from~\eqref{eqn:aux-a} and the assumption $\sum_{q=0}^{\infty}\frac{1}{n^{\ell_q}} = \infty$. Using the fact that $	\lim_{t\rightarrow\infty}\Lambda_{t,0} = 0$, we have that $ {\lim_{t\rightarrow\infty}\Lambda_{t,s} = 0}$, for any $s>0$. Hence, by Proposition~\ref{prop:shrink}, part (i), we conclude that $\{A(t)\}$ is weakly (and thus strongly) ergodic.
\end{proof}

Following similar steps as in Proposition~\ref{prop:shrink} we obtain the following result for sequences of column-stochastic matrices of the form~\eqref{eqn:W}.
\begin{proposition}\label{prop:sh2}
Consider the dynamics~\eqref{eqn:dynamicalsystem} and assume that sequence of matrices $\{A(t)\}$ is strongly aperiodic and has the directed infinite flow property, where the $A(t)$ are weighted adjacency matrices in the form of~\eqref{eqn:W}.				
Then,
\begin{enumerate}[(i)]
\item
there is a vector $\phi(t) \in \real^{n}$ such that, for all $i,j \in [n]$ and $t\geq s$,
\begin{align*}
\biggl|[A(t:s)]_{ij} - \phi_i(t)\biggl| \leq \Lambda_{t,s},
\end{align*}
where $\Lambda_{t,s} = \prod_{q\in \Q_{t,s}}\lambda_q$ and $\lambda_q = \left(1-\frac{1}{n^{\ell_q}}\right)$;
\item
for the sequence $\{\ell_q\}$ associated with $\{A(t)\}$, if
\begin{align*}
\sum_{q=1}^{\infty}\frac{1}{n^{\ell_q}}= \infty,
\end{align*}
then for all $j\in [n]$, $\lim_{t\rightarrow \infty}\biggl|[A(t:s)]_{ij} - \phi_i(t)\biggl| =0$.
\end{enumerate}
\end{proposition}

It is worth pointing out that in Proposition~\ref{prop:shrink}, since the $A(t)$ are row-stochastic, $x(t)$ approaches a vector with identical entries. However, in Proposition~\ref{prop:sh2} the $x(t)$ does not necessarily approach a fixed vector.

\section{Convergence of Push-Sum}\label{sec:convergenceofpushsum}
With all the pieces in place, we are now ready to study the behavior of the push-sum algorithm in a random setting. 
\begin{theorem}\label{push}
Consider the push-sum algorithm~\eqref{eqn:mtxalgorithm} and suppose that the sequence of random column-stochastic matrices $\{W(t)\}$ has the directed infinite flow property, almost surely. Then, we have
\begin{align*}
\left| z_i(t+1)- \bar{x}\right|\leq \frac{2\|x(0)\|_1}{{y_{i}(t+1)}}\Lambda_{t,0},
\end{align*}
where $\Lambda_{t,0} = \prod_{q\in \Q_{t,0}}\lambda_q$ and $\lambda_q= \left(1-\frac{1}{n^{\ell_q}}\right)\in (0,1)$.
\end{theorem}
\begin{proof}
Define
\begin{align*}
D(t:s) \triangleq W(t:s) - \phi(t)\ones_n',
\end{align*}
where $\phi(t)$ is a (random) vector from part (i) of Proposition~\ref{prop:sh2}. In addition, under the push-sum algorithm we have that
\begin{align*}
x(t+1)=W(t:0)x(0),\cr
y(t+1)=W(t:0)y(0),
\end{align*}
for all $t\geq0$. Hence, for every $t\geq 0$ and all $i\in [n]$, we have
\begin{align*}
z_i(t+1) - \bar{x}&= \frac{x_i(t+1)}{y_i(t+1)} - \frac{\ones_n' x(0)}{n}\cr
&=\frac{[W(t:0)x(0)]_i}{[W(t:0)y(0)]_i} - \frac{\ones_n' x(0)}{n}\cr
&=\frac{[D(t:0)x(0)]_i + \phi_i(t)\ones_n' x(0)}{[D(t:0)y(0)]_i + \phi_i(t)\ones_n' y(0)} - \frac{\ones_n' x(0)}{n}.
\end{align*}
Using the fact that $y(0) = \ones_n$ and by bringing the fractions to a common denominator, we have
\begin{align*}
z_i(t+1) - \bar{x}=&\frac{[D(t:0)x(0)]_i + \phi_i(t)\ones_n' x(0)}{[D(t:0)\ones_n]_i + n\phi_i(t)} - \frac{\ones_n' x(0)}{n}\cr
=&\frac{n[D(t:0)x(0)]_i + n\phi_i(t)\ones_n' x(0)}{n([D(t:0)\ones_n]_i + n\phi_i(t))}\cr
&- \frac{[D(t:0)\ones_n]_i\ones_n' x(0)+n\phi_i(t)\ones_n' x(0)}{n([D(t:0)\ones_n]_i + n\phi_i(t))}\cr
=&\frac{n[D(t:0)x(0)]_i + [D(t:0)\ones_n]_i\ones_n' x(0)}{n([D(t:0)\ones_n]_i + n\phi_i(t))}.
\end{align*}
Note that the denominator in the last equation is equal to $ny_i(t+1)$. Hence, for all $i\in[n]$ and $t\geq 1$ we have
\begin{align*}
\left|z_i(t+1) - \bar{x}\right|\leq&\frac{\|x(0)\|_1}{{y_{i}(t+1)}}\left(\max_{j}|[D(t:0)]_{ij}|\right)\cr
&+\frac{|\ones_n' x(0)|}{n{y_{i}(t+1)}}\left(\max_{j}|[D(t:0)]_{ij}|\right)n\cr
=&\frac{|\ones_n' x(0)|+\|x(0)\|_1}{{y_{i}(t+1)}}   \left(\max_{j}\left|[D(t:0)]_{ij}\right|  \right),
\end{align*}
where the inequality follows from the triangle inequality. Since $ |\ones_n'x(0)|\leq \|x(0)\|_1$, we have that
\begin{align*}
\left|z_i(t+1) - \bar{x}\right|\leq \frac{2\|x(0)\|_1}{{y_{i}(t+1)}}\left(\max_{j}\left|[D(t:0)]_{ij}\right|\right).
\end{align*}
Using the upper bound in part (i) of Proposition~\ref{prop:sh2},  we obtain 
\begin{align}\label{eqn:T1}
\left|z_i(t+1) - \bar{x}\right|\leq \frac{2\|x(0)\|_1}{{y_{i}(t+1)}} \Lambda_{t,0}.
\end{align}
\end{proof}
\begin{proposition}\label{prop:pushconv}
Consider the push-sum algorithm~\eqref{eqn:mtxalgorithm} and suppose that the sequence of random column-stochastic matrices $\{W(t)\}$ has the directed infinite flow property, almost surely. Moreover, suppose that the sequence $\{\ell_q\}$ associated with $\{W(t)\}$ satisfies~\eqref{eqn:suminfty}, almost surely. If there exists $\delta>0$, such that for any $t\geq0 $, there is $t'\geq t$ such that $y_i(t')\geq \delta$ for all $i\in [n]$, then  
\begin{align*}
\lim_{t\rightarrow\infty}&\left|z_i(t+1) - \bar{x}\right| = 0,\quad\text{almost surely}.
\end{align*}
\end{proposition}
\begin{remark}
	In the next section we exhibit a class of random matrix sequences $\{W(t)\}$ that satisfy the conditions of Proposition~\ref{prop:pushconv} and thus admit average consensus almost surely.
\end{remark}
\begin{proof}
Proof of this proposition is similar to the proof of Theorem 4.1 in~\cite{FB-VB-PT-JT-MV:10}, where the sequence $\{W(t)\}$ is assumed to be stationary; however, since we do not assume stationarity, we provide a proof. By Proposition~\ref{prop:sh2} part (ii), for any $\varepsilon>0 $ there is a time $t_{\varepsilon}$ such that for all $t\geq t_{\varepsilon}$ and $i\in [n]$,
\begin{align*}
\sum_{j=1}^{n}|[W(t:0)]_{ij}-\frac{1}{n}\sum_{k=1}^{n}[W(t:0)]_{ik}|<{\delta\varepsilon}.
\end{align*}
By assumption, there exists $t_{\varepsilon}'\geq t_{\varepsilon}$ such that $y(t_{\varepsilon}')\geq \delta$, which implies that $	f(t_{\varepsilon}')<{\varepsilon}$, where $f(t)$ is defined as in Lemma~\ref{lemma:WG}. Since by Lemma~\ref{lemma:WG}, $f(t)$ is non-increasing, $f(t)<{\varepsilon}$ for all $t\geq t_{\varepsilon}'$, meaning that $f(t)$ converges to zero as $t\rightarrow\infty$ and hence,
$	\lim_{t\rightarrow\infty}\left|z_i(t+1) - \bar{x}\right| = 0 $, almost surely. 
\end{proof}	
\section{B-Irreducible Sequences}\label{sec:convergencerate}
In this section we characterize a class of random column-stochastic matrices that admits average consensus and we provide a rate of convergence of the push-sum algorithm for this class. To achieve this, we restrict the class of random matrices that we consider; as we will point out later, this restricted class still includes many interesting sequences of random matrices.

In the following discussion, we assume that the push-sum dynamics is generated by a column-stochastic matrix sequence $\{W(t)\}$ where 
\begin{align}\label{eqn:wt}
W_{ij}(t) = \frac{R_{ij}(t)}{\sum_{i=1}^{n}R_{ij}(t)},
\end{align}
for all $ i,j \in [n] $, where $ R_{ij}(t)$ is $ 1 $ with probability $ P_{ij}(t) $, and is $ 0 $ with probability $ 1-P_{ij}(t) $ such that ${\{R_{ij}(t):i,j\in [n],t\geq0\}}$ are independent random variables. In other words, there is a random communication link between node $v_j$ and $v_i$ at time $t$ with probability $P_{ij}(t)$. Note that $\{W(t)\}$ is a sequence of independent random column-stochastic matrices.

Furthermore, for the probability matrix sequence $\{P(t)\}_{t\geq 0}$, we assume that the following holds. 
\begin{assumption}\label{assum:Pt}
	$ \{P (t)\}_{t\geq 0}$ is a sequence of $n\times n$ matrices with $ {P_{ij}(t) \in [0,1] }$. Additionally, we assume that $P_{ii}(t)=1$, for all $ v_i \in \V$. Also, for some constant $\eps>0$, we assume that  $P_{ij}(t)\geq \eps$ for all $i,j\in[n]$ and all $t\geq 0$ such that $P_{ij}(t)\not=0$. Finally, we assume that the sequence $\{P(t)\}_{t\geq 0}$ is $B$-irreducible, i.e.\ for some integer $B>0$, \[\sum_{t'=tB}^{(t+1)B-1}P(t)\]
	is irreducible for all $t\geq 0$.		
\end{assumption}

We next state the main result of this section.
\begin{theorem}\label{theorem:main-rate}
Consider the push-sum algorithm~\eqref{eqn:mtxalgorithm} and let $\{W(t)\}$ be a sequence of random column-stochastic matrices defined by \eqref{eqn:wt}, where $\{P(t)\}$ satisfies Assumption~\ref{assum:Pt}. Let $p = \eps^{2(n-1)}$. Then, for any \sloppy ${t\geq B+\frac{2nB}{p} }$, where $n\geq 2$
\begin{align*}
	\Ex\left[ \ln\left(\left|z_i(t+1) - \bar{x}\right|\right)\right]\leq c_0 -c_1t
	\end{align*}
where 
\begin{align*}
c_0 =&  \ln\left( 2\|x(0)\|_1 \right) + \ln (n)\left(\frac{nB}{p}+B\right) +\ln(15),\cr
c_1 = &- \frac{p}{2nB}\ln \left(1-\frac{1}{n^{\frac{4nB}{p}}}\right).
\end{align*}
\end{theorem}

The proof relies on the following results.
\begin{lemma}\label{lm:PtDIFP}
Let $\{W(t)\}$ be a sequence of random column-stochastic matrices defined by \eqref{eqn:wt}, where $\{P(t)\}$ satisfies Assumption~\ref{assum:Pt}.  Let $ \{k_q\} $ and $\{\ell_q\} $ be the sequences defined, respectively, in~\eqref{def:kq} and~\eqref{def:lq} along each sample path. Then
\begin{enumerate}[(i)]
\item the sequence $\{W(t)\}$ has the directed infinite flow property almost surely, and 
\item for the sequence $\{\ell_q\}$, we have
\begin{align*}
\sum_{q=0}^{\infty}\frac{1}{n^{\ell_q}} = \infty,\quad\text{almost surely.}
\end{align*}
\end{enumerate}
\end{lemma}
\begin{proof}
We start by proving~(i). For any $ t\geq0$, let us define the sequence of events 
\begin{align}\label{eqn:At}
\Ev_t=\Bigl\{\sum_{t'=tB}^{(t+1)B-1}W(t') \text{ is irreducible}\Bigr\}.
\end{align}
Note that for all $t\geq 0$, the events $\{\Ev_t\}_{t\geq 0}$ are independent and that $\Ev_t$ implies $\sum_{t'=tB}^{(t+1)B-1} W_{S\bar{S}}(t')>0$, for any non-trivial $S\subset [n]$. Since $\min_{i,j\in [n]:P_{ij}(t)>0}P_{ij}(t) >\eps>0$, for all $t\geq 0$, we have 
\begin{align*}
\Pr(\Ev_t)\geq\eps^{2(n-1)}.
\end{align*}
This follows from~\cite[Corollary 5.3.6]{digraph} and the fact that $\{P(t)\}$ is \sloppy \mbox{$B$-irreducible} and hence, there is at least a subset of size $2(n-1)$ of the edges $(v_j,v_i)$ that form a strongly connected graph and $P_{ij}(t')\geq \eps$ for some $t'\in [tB,(t+1)B-1]$.

Since the events $\Ev_t$ are independent, hence, by the second Borel-Contelli lemma~\cite[Theorem 2.3.6]{RD:10}, ${\sum_{t'=tB}^{(t+1)B-1}W_{S\bar{S}}(t') >0}$ infinitely often, almost surely. Moreover, since every positive entry of $W(t)$ is bounded below by~$\frac{1}{n}$, for any non-trivial ${S\subset [n]}$, $\sum_{t=0}^{\infty}W_{S\bar{S}}(t) = \infty$, almost surely, implying that $\{W(t)\}$ has the directed infinite flow property, almost surely. This also implies that $k_q$ and $\ell_q$ are finite for all $q$, almost surely. This completes the proof of~(i).

To prove~(ii), let us define, for all $t\geq 0$, 
\begin{align}\label{eqn:Ct}
\Ec_t=\bigcap_{t'=tn}^{(t+1)n-1}\Ev_{t'},
\end{align}
where $\Ev_t$ is defined in \eqref{eqn:At}. Since the $\Ev_t$ are independent, $\Pr(\Ec_t) = \prod_{t'=tn}^{(t+1)n-1} \Pr(\Ev_{t'})\geq \eps^{2n(n-1)}$ for all ${t\geq 0}$. This implies that ${\sum_{t=0}^{\infty}\Pr(\Ec_t) =\infty}$. Again, since the $\Ec_t$ are independent, by the Borel-Contelli lemma, $\Ec_t$ occurs infinitely often, almost surely. This implies that ${\ell_q\leq nB}$ infinitely often, almost surely. Hence, $\sum_{q=1}^{\infty}\frac{1}{n^{\ell_q}} = \infty$, almost surely.
\end{proof}
\begin{lemma}
In the push-sum algorithm~\eqref{eqn:mtxalgorithm} let $\{W(t)\}$ be a sequence of random column-stochastic matrices corresponding to the sequence $\{P(t)\}$ satisfying Assumption~\ref{assum:Pt}. Then for all $t\geq 0$ there exists $t'\geq t$ such that for all $i\in [n]$, $y_i(t')\geq \frac{1}{n^{nB}}$.
\end{lemma}
\begin{proof}
Consider the event $\Ec_t$ defined in \eqref{eqn:Ct}. At any time $\Ec_t$ occurs, by Lemma~\ref{lemma:compgraph}, the product $W(tnB+nB-1:tnB)$ is positive; moreover, by Lemma~\ref{lemma:Asu}, $W_{ij}(tnB+nB-1:tnB)\geq \frac{1}{n^{nB}}$ for all $i,j\in [n]$. Since $W(t)$ is column-stochastic, we have ${W_{ij}(tnB+nB-1:0)\geq \frac{1}{n^{nB}}}$. By Lemma~\ref{lm:PtDIFP}, $\Ec_t$ occurs infinitely often, almost surely; therefore, for all $t\geq 0$ there exists $t'\geq t$ such that for all $i\in [n]$, $y_i(t')\geq \frac{1}{n^{nB}}$.
\end{proof}
The preceding two lemmas and Proposition~\ref{prop:pushconv} imply the following.
\begin{corollary}
Let $\{W(t)\}$ be a sequence of random column-stochastic matrices corresponding to the sequence $\{P(t)\}$ satisfying Assumption~\ref{assum:Pt}. Then $ \{W(t)\} $ admits average consensus, almost surely.
\end{corollary}
\begin{lemma}\label{Lemma:product-rate}
Let $\{W(t)\}$ be a sequence of random column-stochastic matrices corresponding to the sequence $\{P(t)\}$ satisfying Assumption~\ref{assum:Pt}. Let $ \{\ell_q\} $ be the sequence defined in~\eqref{def:lq} along each sample path. For all ${t\geq B + \frac{2nB}{p}}$, we have
\begin{align*}
\Ex\left[\Lambda_{t,0}\right] \leq \exp\left(-\beta_t^2\left(\frac{t}{B}-2\right) \right) + 2\left(1-\frac{1}{n^{\frac{4nB}{p}}}\right)^{ \frac{pt}{2nB}},
\end{align*}
where $\Lambda_{t,0} = \prod_{q\in \Q_{t,0}}(1-\frac{1}{n^{l_q}})$, $\beta_t = \frac{p}{2}-\frac{2pB}{t}$, and ${p=\eps^{2(n-1)}}$.
\end{lemma}
\begin{proof}
Let $X_B(t)$ be the indicator of the event $\Ev_t$, i.e., 
\begin{align*}
X_B(t) = 
\begin{cases}
1 &\text{if }\sum_{t'=tB}^{(t+1)B-1} W(t') \text{ is irreducible},\cr
0 &\text{otherwise.}
\end{cases}
\end{align*}
By the preceding argument, we have $\Pr(X_B(t)=1)\geq p= \eps^{2(n-1)}>0$. Note that the $X_B(t)$ are independent. We let $H_B(T) = \sum_{t=0}^{T}X_B(t)$ for all $T\geq 0$, and define
\begin{align*}
q_t \triangleq \max\{q:k_q\leq t\}.
\end{align*}
By definition of $H_B(\cdot)$ and $q_t$, we have that
\begin{align}\label{eqn:qgeqH}
q_t\geq H_B\left(\left\lfloor\frac{t}{B}\right\rfloor-1\right).
\end{align}
Now, we have that 
\begin{align*}
\Ex \left[\Lambda_{t,0}\right] 
=& \Ex\left[ \Lambda_{t,0}\ \biggl| \ q_t \leq \frac{pt}{2B}  \right] \Pr\left(q_t \leq \frac{pt}{2B}  \right) \\
&+ \Ex\left[\Lambda_{t,0}\ \biggl|\ q_t > \frac{pt}{2B} \right] \Pr\left(q_t > \frac{pt}{2B}\right).
\end{align*}
Since all terms on the right-hand side are less than or equal to $1$, we have
\begin{align*}
\Ex\left[\Lambda_{t,0}\right] \leq \Pr\left(q_t \leq \frac{pt}{2B} \right) + \Ex\left[\Lambda_{t,0}\ \biggl|\ q_t > \frac{pt}{2B} \right].
\end{align*}
Using~\eqref{eqn:qgeqH}, we have
\begin{align*}
\Ex\left[\Lambda_{t,0}\right] \leq& \Pr\left(H_B\left(\left\lfloor\frac{t}{B}\right\rfloor-1\right) \leq \frac{pt}{2B} \right) \\
&+ \Ex\left[\Lambda_{t,0}\ \biggl| \ q_t > \frac{pt}{2B} \right].
\end{align*}
Let us consider the second term on the right-hand side. When $ q_t > \frac{pt}{2B} $, we have $\left| \Q_{t,0}\right|\geq \left\lfloor\frac{pt}{2nB}\right\rfloor$. Using Lemma~\ref{lemma:opt1} to maximize the second term on the right-hand side over the choices of $\ell_q$, we obtain
\begin{align}\label{eqn:partB}
\Ex\left[\Lambda_{t,0}\ \biggl|\ q_t > \frac{pt}{2B} \right]	\leq& \left(1-\frac{1}{n^{\frac{t}{\left\lfloor \frac{pt}{2nB}\right\rfloor}}}\right)^{\left\lfloor \frac{pt}{2nB}\right\rfloor}\cr \leq& 2\left(1-\frac{1}{n^{\frac{t}{\left\lfloor \frac{pt}{2nB}\right\rfloor}}}\right)^{ \frac{pt}{2nB}}.
\end{align}

To further simplify the above inequality, we show that $\frac{t}{\left\lfloor \frac{pt}{2nB}\right\rfloor}\leq \frac{4nB}{p}$. To show this, we note that for all $t\geq\frac{2nB}{p}+B$, we have $\frac{pt}{2nB}>1$ and hence, ${\left\lfloor \frac{pt}{2nB}\right\rfloor}\geq 1$. Now, assume that $\xi={\left\lfloor \frac{pt}{2nB}\right\rfloor}\geq 1$. We have $2nB\xi\leq pt\leq 2nB(\xi+1)$. Therefore, 
\begin{align*}
	\frac{t}{\left\lfloor \frac{pt}{2nB}\right\rfloor}\leq \frac{2nB}{p}\left(\frac{\xi+1}{\xi}\right)\leq \frac{4nB}{p},
\end{align*}
where the last inequality follows from the fact that $\xi\geq 1$. 

Using this inequality in \eqref{eqn:partB}, we get
\begin{align}\label{eqn:rhsfirst}
\Ex\left[\Lambda_{t,0}\ \biggl|\ q_t > \frac{pt}{2B} \right]	\leq&2 \left(1-\frac{1}{n^{\frac{t}{\left\lfloor \frac{pt}{2nB}\right\rfloor}}}\right)^{ \frac{pt}{2nB}}\cr 
\leq& 2\left(1-\frac{1}{n^{\frac{4nB}{p}}}\right)^{ \frac{pt}{2nB}}.
\end{align}
On the other hand, since $\Ex[X_B(t)]\geq p$ for all $t\geq B$, we have
\begin{align*}
&\Pr\left(H\left(\left\lfloor\frac{t}{B}\right\rfloor-1\right) \leq \frac{pt}{2B} \right) \cr
&= \Pr\left(\sum_{t'=0}^{\lfloor{t/B}\rfloor-1}X_B(t') -  p \left(\left\lfloor\frac{t}{B}\right\rfloor-1\right) \leq - \alpha_t \left(\left\lfloor\frac{t}{B}\right\rfloor-1\right)\right)\cr
&\leq \Pr\left(\sum_{t'=0}^{\lfloor{t/B}\rfloor-1}\left(X_B(t') -  \Ex[X_B(t')] \right)\leq -\alpha_t \left(\left\lfloor\frac{t}{B}\right\rfloor-1\right)\right),
\end{align*}
where
\begin{align}\label{eqn:alpha}
\alpha_t = \frac{p\left(\left\lfloor\frac{t}{B}\right\rfloor-1\right)-\frac{pt}{2B}}{\left\lfloor\frac{t}{B}\right\rfloor-1}.
\end{align}
When $t\geq B+\frac{2nB}{p}$, $\alpha_t>0$ and hence, by Lemma~\ref{lemma:Hoeffding}, we obtain
\begin{align}\label{eqn:rhssecond}
\Pr\left(H\left(\left\lfloor\frac{t}{B}\right\rfloor-1\right) \leq \frac{pt}{2B} \right) \leq& \exp\left(-\alpha_t^2\left(\left\lfloor\frac{t}{B}\right\rfloor-1\right)\right)\cr 
\leq& \exp\left(-\alpha_t^2\left(\frac{t}{B}-2\right)\right).
\end{align}
From~\eqref{eqn:alpha}, we have 
\begin{align*}
\alpha_t > \frac{p\left(\frac{t}{B}-2\right)-\frac{pt}{2B}}{\frac{t}{B}}=\frac{p}{2}-\frac{2pB}{t}.
\end{align*}
If we let $\beta_t =\frac{p}{2}-\frac{2pB}{t}$, using~\eqref{eqn:rhsfirst} and \eqref{eqn:rhssecond}, we conclude that
\begin{align*}
\Ex\left[\Lambda_{t,0}\right] \leq \exp\left(-\beta_t^2\left(\frac{t}{B}-2\right) \right) + 2\left(1-\frac{1}{n^{\frac{4nB}{p}}}\right)^{ \frac{pt}{2nB}},
\end{align*}
finishing the proof.
\end{proof}

\begin{lemma}\label{Lemma:El1}
In the push-sum algorithm~\eqref{eqn:mtxalgorithm} let $\{W(t)\}$ be a sequence of random column-stochastic matrices corresponding to the sequence $\{P(t)\}$ satisfying Assumption~\ref{assum:Pt}. We have, for all $i\in [n]$ and $t\geq 0$,
\begin{align*}
\Ex\left[\ln\left(\frac{1}{y_{i}(t)}\right)\right]\leq \ln(n)\left(B\frac{n}{p}+B\right).
\end{align*}
\end{lemma}
\begin{proof}
By Lemma~\ref{lemma:Asu}, for all $t< \frac{Bn}{p}+B$ and $i\in [n]$,
\begin{align*}
[W(t:0)]_{ii}\geq \frac{1}{n^{B\frac{n}{p}+B}} ,
\end{align*}
almost surely. This implies that
\begin{align*}
\Ex\left[\ln\left(\frac{1}{y_i(t)}\right)\right]\leq \ln(n)\left(B\frac{n}{p}+B\right),
\end{align*}
for all $t< \frac{Bn}{p}+B$ and $i\in [n]$.
If $t\geq \frac{Bn}{p}+B$, let $t = aB+b$, where $a, b\in \integers_{\geq 0}$ and $b<B$. Define
\begin{equation*}
\tau_t = \begin{cases}
\min\{T:\sum_{t=a-T}^{a-1} X_B(t)=n\}, & \text{if }\sum_{t=0}^{a-1}X_B(t)\geq n\\
a& \text{otherwise}.
\end{cases}
\end{equation*}
When $\tau_t=a$, ${W_{ij}(t:0)\geq \frac{1}{n^{\tau_t B+B}}}$, for all $i,j\in [n]$. When $\tau_t\neq a$, by Lemma~\ref{lemma:compgraph}, ${W(aB-1:(a-\tau_t) B)}$ is a positive matrix and consequently by Lemma~\ref{lemma:Asu}, ${W_{ij}(t:(a-\tau_t) B)\geq \frac{1}{n^{\tau_t B+B}}}$ for all ${i,j\in [n]}$; in addition, since the $W(t)$ are column-stochastic, we have ${W_{ij}(t:0)\geq \frac{1}{n^{\tau_t B+B}}}$. Therefore, for all $t\geq 0$ we have
\begin{align*}
\ln\left(\frac{1}{W_{ij}(t:0)}\right)\leq \ln(n)(\tau_t B+B)\quad\text{for all } i,j\in[n].
\end{align*}
Consider a sequence of independent Bernoulli trials $Y_t$, where in each trial the probability of success is $p$. The number of trials until $n$ successes occur is a negative binomial random variable $Z$ having parameters $n$ and $p$. Since ${\Pr(\tau_t \leq i)\geq\Pr(Z\leq i)}$ for all $i\geq n $, we have $\Ex[\tau_t]\leq \Ex[Z]$. Since ${\Ex[Z] = \frac{n}{p}}$, we obtain
$		\Ex[\tau_t]\leq\frac{n}{p} $,
and hence the result follows. 
\end{proof}	

We are now in a position to prove Theorem~\ref{theorem:main-rate}.

\begin{proof}[Proof of Theorem~\ref{theorem:main-rate}]
In \eqref{eqn:T1}, since both sides are positive, we have
\begin{align*}
\ln\left(\left|z_i(t+1) - \bar{x}\right|\right)\leq& \ln\left(\frac{2\|x(0)\|_1}{{y_{i}(t+1)}} \Lambda_{t,0}\right)\cr
=& \ln\left( 2\|x(0)\|_1 \right) + \ln\left(\frac{1}{{y_{i}(t+1)}} \right) \cr 
&+ \ln\left(\Lambda_{t,0}\right).
\end{align*}
By taking expectations and using Lemma~\ref{Lemma:El1}, we obtain
\begin{align}\label{eqn:Eineq}
\Ex\left[ \ln\left(\left|z_i(t+1) - \bar{x}\right|\right)\right]\leq& \ln\left( 2\|x(0)\|_1 \right) + \ln (n)\left(\frac{nB}{p}+B\right)\cr 
&+ \Ex \left[\ln\left(\Lambda_{t,0}\right)\right]\cr 
\leq& \ln\left( 2\|x(0)\|_1 \right) + \ln (n)\left(\frac{nB}{p}+B\right)\cr 
&+ \ln\left(\Ex \left[\Lambda_{t,0}\right]\right),
\end{align}
where the last inequality follows from Jensen's inequality. Now by Lemma~\ref{Lemma:product-rate}, we have
\begin{align*}
\Ex \left[\Lambda_{t,0}\right]\leq& \exp\left(-\beta_t^2\left(\frac{t}{B}-2\right)\right) +2\left(1-\frac{1}{n^{\frac{4nB}{p}}}\right)^{\frac{pt}{2nB}},
\end{align*}
where $\beta_t = \frac{p}{2}-\frac{2pB}{t}$. Let us consider the first term on the right hand side; since $\beta_t\leq \frac{1}{2}$ we have
\begin{align*}
\exp\left(-\beta_t^2\left(\frac{t}{B}-2\right)\right)\leq& \exp\left(-\beta_t^2\frac{t}{B}+\frac{1}{2}\right) \cr 
= & \exp\left(-\frac{p^2t}{4B}+2p^2+\frac{1}{2}-\frac{4p^2B}{t}\right) \cr 
\leq & \exp\left(-\frac{p^2t}{4B}+\frac{5}{2}\right) \cr 
\leq & 13\exp\left(-\frac{p^2t}{4B}\right) \cr 
=& 13\left(\exp\left(-\frac{pn}{2}\right)\right)^{\frac{pt}{2nB}}.
\end{align*}
Since $n\geq 2$, $\exp\left(-\frac{pn}{2}\right) \leq \exp\left(-p\right)$. On the other hand, ${\left(1-\frac{1}{n^{\frac{4nB}{p}}}\right) \geq \left(1-\frac{1}{2^{\frac{8}{p}}}\right) }$ for all $n\geq 2 $ and $B\geq 1$. It can be seen the for $p\in [0,1]$,  $\exp\left(-p\right) \leq \left(1-\frac{1}{2^{\frac{8}{p}}}\right)  $, and consequently $ {\exp\left(-\frac{pn}{2}\right) \leq \left(1-\frac{1}{n^{\frac{4nB}{p}}}\right)}$. Hence
\begin{align}\label{eqn:Elambda}
\Ex \left[\Lambda_{t,0}\right]\leq 15 \left(1-\frac{1}{n^{\frac{4nB}{p}}}\right)^{\frac{pt}{2nB}}.
\end{align}
The result now follows using \eqref{eqn:Eineq} and \eqref{eqn:Elambda}. 
\end{proof}

\section{Conclusion}\label{sec:conclusion}
We have studied the convergence properties of the push-sum algorithm for average consensus on sequences of random directed graphs. We have proved that this dynamics is convergent almost surely when some mild connectivity assumptions are met and the auxiliary states of agents are uniformly bounded away from zero infinitely often. We have shown that the latter assumption holds for sequences of random matrices constructed using a sequence of time-varying $ B $-irreducible probability matrices. We have also obtained convergence rates for the proposed push-sum algorithm. Future work include studying the implications in scenarios with link-failure and in distributed optimization on random time-varying graphs. 


\section{Appendix}
\newcounter{mycounter}
\renewcommand{\themycounter}{A.\arabic{mycounter}}
\newtheorem{propositionappendix}[mycounter]{Proposition}
\newtheorem{lemmaappendix}[mycounter]{Lemma}
\newtheorem{remarkappendix}[mycounter]{Remark}

\begin{lemmaappendix}\label{lemma:compgraph}
	For $n\geq 2$, let $\{A(i)\}_{i=1}^{n-1}$ be a sequence of weighted adjacency matrices associated with the strongly connected directed graphs $\{\G(i)\}_{i=1}^{n-1}$ on the node set ${\V = \{v_1,v_2,\ldots,v_n\}}$, where ${\G(i) = (\V,\Es(i),A(i))}$ and $A(i)\in S_n^{+}$ for all $i\in [n-1]$. Then the matrix product $A(n-1:1)$ is positive. 
\end{lemmaappendix}
\begin{proof}
	Let $\G(k:1)= (\V,\Es(k:1))$ indicate the directed graph associated with the product $A(k:1)$, where $k\in [n-1]$. Let $\Nout_i(k:1)$ and $\dout_i(k:1)$ indicate the set of out-neighbors and out-degree of node $i\in [n]$ in directed graph $\G(k:1)$, respectively. Consider an arbitrary but fixed node $i\in [n]$. Since $A(1)\in \S^{+}_n$ and $\G(1)$ is strongly connected, we have
	\begin{align}\label{eqn:d1}
	\dout_i(1)\geq 2.
	\end{align}
	Now consider the directed graph $\G(k:1)$ and assume that ${\dout_i(k:1)\leq n-1}$ for some $k\in[n-1]$; we show that ${\dout_i(k+1:1)>\dout_i(k:1)}$. By Lemma~\ref{lemma:Asu}(ii), we have ${\Nout_i(k:1)\subseteq \Nout_i(k+1:1)}$. Moreover, since $\G(k+1)$ is strongly connected and $\dout_i(k:1)\leq n-1$, there is ${l\notin \Nout_i(k:1)}$ such that $l\in \Nout_j(k+1)$ for some $j\in \Nout_i(k:1)$; otherwise, there is no path between $i$ and $l$ in $\G(k+1)$, contradicting the strong connectivity of $\G(k+1)$. Hence, by Lemma~\ref{lemma:Asu} (iii) $l\in \Nout_i(k+1:1)$, implying that \begin{align*}\label{eqn:dk+1}
	\dout_i(k+1:1)>\dout_i(k:1).
	\end{align*} 
	This along with \eqref{eqn:d1} imply that
	\begin{align*}
	\dout_i(k:1)\geq k+1,,
	\end{align*} 
	for all $k\in [n-1]$, which implies that $\dout_i(n-1:1) = n$. Since this statement holds for any $i\in [n]$, the matrix product $A(n-1:1)$ is positive.
\end{proof}

\begin{lemmaappendix}[Lemma 1 \cite{AN-AO:09}]\label{lemma:Asu}
Consider a sequence of directed graphs $\{\G(t)\}$, which we assume to contain all the self-loops, with a corresponding sequence of weighted adjacency matrices $\{A(t)\}$. In addition, assume that $A_{ij}(t)\geq \gamma$ whenever $A_{ij}(t)>0$, for some $\gamma>0$. 
Then the following statements hold:
\begin{enumerate}[(i)]
\item $[A(t:s)]_{ii}\geq \gamma^{t-s+1}$, for all $i\in[n]$ and $t\geq s\geq 0$;
\item if $[A(r)]_{ij}>0$ for some $t\geq r\geq s\geq 0$ and $i,j\in[n]$, then ${[A(t:s)]_{ij}\geq \gamma^{t-s+1}}$;
\item if $[A(s)]_{ik}>0$ and $[A(r)]_{kj}>0$ for some ${t\geq r>s\geq 0}$, then $[A(t:s)]_{ij}\geq \gamma^{t-s+1}$.
\end{enumerate}
\end{lemmaappendix}

\begin{lemmaappendix}[Lemma 4.3 \cite{FB-VB-PT-JT-MV:10}]\label{lemma:WG}
Consider the push-sum algorithm \eqref{eqn:mtxalgorithm}. Define 
\begin{align*}
f(t) = \max_{i\in [n]}\frac{\sum_{j=1}^{n}|[W(t:0)]_{ij}-\frac{1}{n}\sum_{k=1}^{n}[W(t:0)]_{ik}|}{y_i(t)}.
\end{align*}
Then, $f(t)$ is non-increasing and 
\begin{align*}
\|z(t)-\bar{x}\ones_n\|_{\infty}\leq \|x(0)\|_{\infty}f(t).
\end{align*}
\end{lemmaappendix}

\begin{lemmaappendix}[Hoeffding's inequality~\cite{WH:63}]\label{lemma:Hoeffding}
If $X_1,X_2,\ldots,X_n $ are independent random variables and $0\leq X_i\leq 1$, for all $i\in [n]$, then for any $\alpha>0$, we have
\begin{align*}
\Pr\left(\sum_{i=1}^{n}\left(X_i - \Ex[X_i]\right)\leq -\alpha n\right)\leq \exp\left(-2\alpha^2n\right).
\end{align*}
\end{lemmaappendix}

\begin{lemmaappendix}\label{lemma:opt1}
For $n>1$ and for all $l_1, l_2, \ldots, l_q\in \integers_{\geq 0}$, ${q>0}$, we have
\begin{align*}
\prod_{i=1}^{q}\left(1-\frac{1}{n^{l_i}}\right)\leq\left(1-\frac{1}{n^{\frac{t}{q}}}\right)^{q},
\end{align*}
where $ t= l_1+l_2+\cdots +l_q$.
\end{lemmaappendix}
\begin{proof}
It suffices to show that
\begin{align*}
\frac{1}{q}\sum_{i=1}^{q}\ln\left(1-\frac{1}{n^{l_q}}\right)\leq\ln\left(1-\frac{1}{n^{\frac{t}{q}}}\right),
\end{align*}
which simply follows from Jensen's inequality, since the function ${g(\zeta) = \ln\left(1-\frac{1}{n^{\zeta}}\right)}$ is concave.
\end{proof}

\bibliographystyle{ieeetr}%
\bibliography{alias,PR-add,BG-add,Main,Main-add,JC,BG}

\end{document}